\pdfoutput=1
\documentclass[12pt,a4paper,reqno]{amsart}
\usepackage{amssymb}
\usepackage{amscd}
\usepackage{graphicx}
\numberwithin{equation}{section}

\theoremstyle{plain}

\newtheorem{theorem}{Theorem}[section]
\newtheorem{proposition}[theorem]{Proposition}
\newtheorem{lemma}[theorem]{Lemma}
\newtheorem{corollary}[theorem]{Corollary}
\newtheorem{conjecture}[theorem]{Conjecture}

\theoremstyle{definition}

\newcommand\R{\mathbb{R}}
\newcommand\Z{\mathbb{Z}}
\newcommand\N{\mathbb{N}}

\begin{document}

\title{Deterministic methods to find primes}

\author{D.H.J. Polymath}
\address{http://michaelnielsen.org/polymath1/index.php}


\begin{abstract}  Given a large positive integer $N$, how quickly can one construct a prime number larger than $N$ (or between $N$ and $2N$)?  Using probabilistic methods, one can obtain a prime number in time at most $\log^{O(1)} N$ with high probability by selecting numbers between $N$ and $2N$ at random and testing each one in turn for primality until a prime is discovered.  However, if one seeks a deterministic method, then the problem is much more difficult, unless one assumes some unproven conjectures in number theory; brute force methods give a $O(N^{1+o(1)})$ algorithm, and the best unconditional algorithm, due to Odlyzko, has a runtime of $O(N^{1/2 + o(1)})$.

In this paper we discuss an approach that may improve upon the $O(N^{1/2+o(1)})$ bound, by suggesting a strategy to determine in time $O(N^{1/2-c})$ for some $c>0$ whether a given interval in $[N,2N]$ contains a prime.  While this strategy has not been fully implemented, it can be used to establish partial results, such as being able to determine the \emph{parity} of the number of primes in a given interval in $[N,2N]$ in time $O(N^{1/2-c})$.
\end{abstract}

\maketitle



\section{Introduction}

We\footnote{A list of people involved in this Polymath project can be found at {\tt michaelnielsen.org/polymath1/index.php?title=Polymath4\_grant\_acknowledgments}.} consider the following question: given a large integer $N$, how easy is it to generate a prime number that is larger than $N$?

Of course, since there are infinitely many primes, and each integer can be tested for primality in finite time, one can always answer this question in finite time, simply by the brute force method of testing each integer larger than $N$ in turn for primality.  So the more interesting question is to see how rapidly one can achieve this, and in particular to see for which $A = A(N)$ is it possible for a Turing machine (say) to produce a prime number larger than $N$ in at most $A$ steps and using at most $A$ units of memory, taking only the integer $N$ as input.  If $A$ is such that this task is possible, we say that a prime number larger than $N$ can be found ``in time at most $A$''.

Note that if one allows probabilistic algorithms (so that the Turing machine also has access to a random number generator for input), then one can accomplish this in time polynomial in the length of $N$ (i.e. in time at most $\log^{O(1)} N$); indeed, one can select integers in $[N,2N]$ at random and test each one for primality.  (Here we use the usual asymptotic notation, thus $O(X)$ denotes a quantity bounded in magnitude by $CX$ where $C$ is independent of $N$, and $o(1)$ denotes a quantity bounded in magnitude by $c(N)$ for some $c(N)$ going to zero as $N \to \infty$.)
Using algorithms such as the AKS algorithm \cite{aks}, each such integer can be tested in time at most $\log^{O(1)} N$, and by the prime number theorem one has about a $1/\log N$ chance of success with each attempt, so the algorithm will succeed with (say) $99\%$ certainty after at most $\log^{O(1)} N$ units of time.  See also \cite{pss} for a probabilistic algorithm for detecting primes in polynomial time which predates the AKS algorithm, as well as a deterministic algorithm for detecting primes in a certain subset of the primes.

If however one insists on \emph{deterministic} methods, then the problem becomes substantially harder.  The sieve of Eratosthenes will supply all the primes between $N$ and $2N$, but requires $O(N^{1+o(1)})$ units of time and memory.  Using the AKS algorithm, if one can construct a subset $E$ of $[N,2N]$ in time at most $A$ that is guaranteed to contain at least one prime, then by testing each element of $E$ in turn for primality, we see that we can obtain a prime in time at most $A + O( N^{o(1)} |E| )$.   Thus, for instance, using Bertrand's postulate one recovers the $O(N^{1+o(1)})$ bound; using the unconditional fact that $[N,N+N^{0.525}]$ contains a prime for every large $N$ (see \cite{bhp}) we improve this to $O(N^{0.525+o(1)})$; and if one assumes the Riemann hypothesis, then as is well known we obtain a prime in an interval of the form $[N,N+N^{0.5+o(1)}]$ for all large $N$, leading to a bound of $O(N^{0.5+o(1)})$.

There are other sparse sets that are known to contain primes.  For instance, using the result of Heath-Brown \cite{heath} that there are infinitely many primes of the form $a^3+2b^3$ (which comes with the expected asymptotic), the above strategy gives an unconditional algorithm with time $O( N^{2/3 + o(1)} )$, since the number of integers in $[N,2N]$ of the form $a^3+2b^3$ is comparable to $N^{2/3}$.  More generally, if one assumes Schinzel's hypothesis $H$, which predicts the asymptotic number of primes inside any polynomial sequence $\{ P(n): n \in \N \}$, and in particular inside the sequence $n^k+1$ for any fixed $k = 1,2,\ldots$, then the same argument would give a deterministic prime-finding algorithm that runs in time $O(N^{1/k+o(1)})$.  Unfortunately the asymptotic for primes of the form $n^k+1$ is not known even for $k=2$, which is a famous open conjecture of Landau.

A famous conjecture of Cram\'er \cite{cramer} (see also \cite{granville} for refinements) asserts that the largest prime gap in $[N,2N]$ is of the order of $O(\log^2 N)$, which would give a deterministic algorithm with run time $O( \log^{O(1)} N )$.  Unfortunately, this conjecture is also well out of reach of current technology (the best bound on prime gaps being the $O(N^{0.525+o(1)})$ result from \cite{bhp} mentioned earlier, or $O(\sqrt{N \log N})$ assuming the Riemann hypothesis \cite{cramer}).

Another way to proceed is to find an efficient way to solve the following \emph{decision problem}: given a subinterval $[a,b]$ of $[N,2N]$, how quickly can one decide whether such an interval contains a prime?  If one could solve each such problem in time at most $A$, then one could locate a prime in $[N,2N]$ in time $O( A \log N )$, by starting with the interval $[N,2N]$ (which is guaranteed to contain a prime, by Bertrand's postulate) and then performing a binary search, repeatedly subdividing the interval into two approximately equal pieces and using the decision problem to locate a subinterval that also contains a prime.

Because primality testing is known to be in the complexity class $P$ (see \cite{aks}), we see that the above decision problem is in the complexity class $NP$.  Thus, if $P=NP$, we could locate a prime deterministically in time at most $\log^{O(1)} N$.   Of course, this conjecture is also unproven (and is widely believed to be false).

Given that there is a probabilistic algorithm to locate primes in time polynomial in the digits, it may seem that the conjecture $P=BPP$ would be able to quickly imply a fast algorithm to locate primes.  Unfortunately, to use the $P=BPP$ conjecture, one would presumably need to obtain a bounded-error probabilistic polynomial (BPP) time algorithm for solving the above decision problem (or some closely related problem), and it is not clear how to achieve this\footnote{For further discussion of this issue, see 

{\tt michaelnielsen.org/polymath1/index.php?title=Oracle\_counterexample\_to\_finding\_pseudoprimes} .}.

One way to solve the above decision problem would be to find a quick way to compute $\pi(x)$, the number of primes less than or equal to $x$, for $x$ in $[N,2N]$, since an interval $[a,b]$ contains a prime if and only if $\pi(b)-\pi(a-1) > 0$.  The fastest known elementary method to compute $\pi(x)$ is the Meissel-Lehmer method \cite{lmo}, \cite{rivat}, which takes time $O(x^{2/3}/\log^2 x)$ and leads to a $O(N^{2/3+o(1)})$ algorithm.  

On the other hand, if one can calculate $\pi(x)$ for $x \in [N,2N]$ \emph{approximately} in time $A$ to a guaranteed error of $L$ (say), then a modification of the above arguments shows that in time $O( N^{o(1)} A)$, one can find a subinterval of $[N,2N]$ of length  $O(N^{o(1)} L)$.  (The only thing one has to be careful of is to ensure in the binary search algorithm that the density of primes in the interval is always $\gg 1/\log N$, but this is easily accomplished.)   It was observed by Lagarias and Odlyzko \cite{lag} that by using an explicit contour integral formula for $\pi(x)$ (or the closely related expression $\psi(x) = \sum_{n \leq x}\Lambda(n)$) in terms of the Riemann zeta function, one could compute $\pi(x)$ to accuracy $L$ using $O( N^{o(1)} \frac{N}{L} )$ time\footnote{The basic idea is to use quadrature to integrate a suitable contour integral involving the zeta function on the interval from $2-iT$ to $2+iT$, where $T$ is comparable to $N^{o(1)} \frac{N}{L}$.  In \cite{lag} it is also observed that the method also lets one compute $\pi(x)$ exactly in time $O(N^{1/2+o(1)})$, by smoothing the sum $\psi(x)$ at scale $O(N^{1/2+o(1)})$ and using the sieve of Eratosthenes to compute exactly the error incurred by such a smoothing.}.  This is enough to obtain an interval of length $O(N^{1/2+o(1)})$ that is guaranteed to contain a prime, in time $O(N^{1/2+o(1)})$; testing each such element for primality, one then obtains a deterministic prime-finding algorithm that unconditionally takes $O(N^{1/2+o(1)})$ time (thus matching the algorithm that was conditional on the Riemann hypothesis).  To our knowledge, this is the best known algorithm in the literature for deterministically finding primes.

\subsection{Beating the square root barrier?}

We conjecture that the square root barrier for the decision problem can be broken:

\begin{conjecture}\label{break}  There exists an absolute constant $c>0$, such that one can (deterministically) decide whether a given interval $[a,b]$ in $[N,2N]$ of length at most $N^{1/2+c}$ contains a prime in time $O(N^{1/2-c+o(1)})$.
\end{conjecture}

This would of course imply a bound of $O(N^{1/2-c+o(1)})$ for finding a prime in $[N,2N]$ deterministically, since as mentioned earlier we can locate an initial interval of length at most $N^{1/2+c}$ containing a prime in time $O(N^{1/2-c+o(1)})$, and then proceed by binary search.

As mentioned earlier, it would suffice to be able to compute $\pi(x)$ in time $O(x^{1/2-c+o(1)})$.  We do not know how to accomplish this, but we have the following partial result:

\begin{theorem}[Computing the parity of $\pi(x)$]\label{parity}  There exists an absolute constant $c>0$, such that one can (deterministically) decide whether a given interval $[a,b]$ in $[N,2N]$ of length at most $N^{1/2+c}$ contains an \emph{odd number} of primes in time $O(N^{1/2-c+o(1)})$.
\end{theorem}

We prove this result in Section \ref{parity-sec}; the key observation is that the parity of the prime counting function $\pi(x)$ is closely connected to the divisor sum function $\sum_{n \leq x} \tau(n)$, which will be computed efficiently by invoking the standard Dirichlet hyperbola identity
\begin{equation}\label{hyperbola}
\begin{split}
\sum_{n \leq x} \sum_{d|n} f(d) g(\frac{n}{d}) &= \sum_{n,m: nm \leq x} f(n) g(m) \\
&= \sum_{n \leq y} g(n) F(\frac{x}{n}) + \sum_{m \leq x/y} f(m) G(\frac{x}{m}) - F(y) G(x/y)
\end{split}
\end{equation}
for any functions $f, g: \N \to \R$, where $F(x) := \sum_{n \leq x} f(n)$ and $G(x) := \sum_{m \leq x} g(m)$; see for instance \cite[\S 3.2, Theorem 1]{tenenbaum}.

Note that once one has Theorem \ref{parity}, and assuming that one can find an interval $[a,b]$ which contains an odd number of primes, then the binary search method will locate a prime deterministically in time $O(N^{1/2-c+o(1)})$, since if one subdivides an interval containing an odd number of primes into two subintervals, then at least one of these must also contain an odd number of primes.  However, we do not know of a method to quickly and deterministically locate an interval with an odd number of primes.

In fact we can establish the following stronger result.  Given an interval $[a,b]$, we define the \emph{prime polynomial} $P(t) = P_{a,b}(t)$ as
$$ P_{a,b}(t) := \sum_{a \leq p\leq b} t^p,$$
where $p$ ranges over primes in $[a,b]$.  Thus for instance $[a,b]$ contains a prime if and only if $P(1)$ is non-zero, or equivalently if $P(t) \hbox{ mod } 2$ is non-zero, where we view $P(t) \hbox{ mod } 2$ as an element of the polynomial ring ${\bf F}_2[t]$ over the field ${\bf F}_2$ of two elements.  

Given a polynomial $P(t)$ over a ring $R$, we say that $P$ has \emph{circuit complexity} $O(M)$ if, after time $O(M)$, one can build a circuit of size $O(M)$ consisting of the arithmetic operations (addition, subtraction, multiplication, and division\footnote{Traditionally, division is not considered an arithmetic operation for the purpose of circuit complexity, but it is convenient for us to modify the definition because we will be taking advantage of division at a few places in the paper.  Also note that in our definition, it is not enough for a circuit to merely exist; it must also be \emph{constructible} within the specified amount of time.}), as well as the primitive polynomials $1, t$, whose output is well-defined in $R[t]$ and is equal to $P$.

\begin{theorem}\label{circo}  Suppose that $[a,b]$ is an interval in $[N,2N]$ of size at most $N^{1/2+c}$ for some sufficiently small $c$.  Then the polynomial $P_{a,b}(t) \hbox{ mod } 2$ has circuit complexity $O(N^{1/2-c+o(1)})$.
\end{theorem}

We prove this theorem in Section \ref{circuit}.

Observe that if $g \in {\bf F}_2[t]$ is a polynomial of degree at most $N^{c/2+o(1)}$, then any arithmetic operation in the quotient space ${\bf F}_2[t] / (g)$ can be performed in time $O(N^{c/2+o(1)})$ (using the fast multiplication algorithm to evaluate multiplication in this space, and Euler's theorem and the power method to perform multiplicative inversion).  As a consequence of this and the above theorem, we see that $P_{a,b}(t) \hbox{ mod } (2, g)$ can be computed in time $O( N^{1/2-c/2+o(1)} )$.  When $g(t) = t-1$, this is Theorem \ref{parity}.  But this theorem is more general.  For instance, applying the above argument with $g$ equal to a cyclotomic polynomial, it is not difficult to see that one can compute the parity of the reduced prime counting functions $\pi(x;a,q) := |\{ p \leq x: p \equiv a \hbox{ mod } q \}|$ for any positive integer $q = O(N^{c/10})$ in time $O(N^{1/2-c/4+o(1)})$.  Unfortunately, we were not able to use this to unconditionally establish Conjecture \ref{break}; it is \emph{a priori} conceivable (but quite unlikely) that an interval $[a,b]$ might contain a non-zero number of primes, but have an even number of primes in every residue class mod $q$  with $q = O(N^{c/10})$.

On the other hand, as the prime polynomial $P_{a,b}(t) \hbox{ mod } 2$ has degree $O(N)$, it is easy to see that the proportion of polynomials of degree at most $N^{c/4}$ that do not divide $P_{a,b}(t) \hbox{ mod } 2$ is bounded away from zero.  (Indeed, a positive proportion of such polynomials contain a prime factor of degree at least $N^{c/8}$, but by unique factorization, there are $O(N)$ such primes, and each one only divides at most $2^{-N^{c/8}}$ of the polynomials of degree at most $N^{c/4}$.)  As such, we see that we can obtain a bounded-error probabilistic algorithm for solving the decision problem that runs in time $O(N^{1/2-c/2+o(1)})$, by testing whether the prime polynomial $P_{a,b}(t)$ vanishes modulo $2$ and $g(t)$, where $g$ is a randomly selected polynomial of degree at most $N^{c/4}$.  Unfortunately, the run time of this algorithm is not polynomial in the number of digits, and so the $P=BPP$ hypothesis does not yield any improvements over existing algorithms.

In Section \ref{extend} we discuss possible strategies that could lead to a full resolution of Conjecture \ref{break}.

\subsection{About this project}

This paper is part of the \emph{Polymath project}, which was launched by Timothy Gowers in February 2009 as an experiment to see if research mathematics could be conducted by a massive online collaboration.  This project (which was administered by Terence Tao) is the fourth project in this series.  Further information on this project can be found on the web site \cite{polywiki}.  Information about this specific polymath project may be found at

\centerline{\tt michaelnielsen.org/polymath1/index.php?title=Finding\_primes}

and a full list of participants and their grant acknowledgments may be found at

\centerline{\tt michaelnielsen.org/polymath1/index.php?title=Polymath4\_grant\_acknowledgments}

We thank Ryan Williams and Tom\'as Oliveira e Silva for corrections, Jeffrey Shallit for a reference, and the anonymous referee for many cogent suggestions and corrections.

\section{Computing the parity of $\pi(x)$}\label{parity-sec}

We now prove Theorem \ref{parity}.  Let $c > 0$ be a small number to be chosen later.   Let $\tau(n) := \sum_{d|n} 1$ be the number of divisors of $n$, and let $\omega(n) := \sum_{p|n} 1$ be the number of distinct primes that divide $n$ (with the convention that $\omega(1)=0$).  One easily verifies the identity
\begin{equation}\label{omn}
2^{\omega(n)} = \sum_{d: d^2|n} \mu(d) \tau(n/d^2)
\end{equation}
where $\mu$ is the M\"obius function\footnote{The \emph{M\"obius function} is defined by setting $\mu(p_1 \ldots p_k) := (-1)^k$ for any product $p_1\ldots p_k$ of distinct primes $p_1,\ldots,p_k$, and $\mu(n)=0$ whenever $n$ is not square-free (i.e. it is divisible by a perfect square larger than $1$).}, by checking this first on prime powers and then using multiplicativity.  Now for $n>1$, $2^{\omega(n)}$ is divisible by $4$, except when $n$ is a prime power $n=p^j$, in which case it is equal to $2$.  This gives the identity
$$ \sum_{a \leq n \leq b} 2^{\omega(n)}  \equiv 2 \sum_{j=1}^\infty |\{ p \in [a^{1/j},b^{1/j}]: p \hbox{ prime}\}| \hbox{ mod } 4.$$
Clearly we may restrict $j$ to size $O(\log N)$.  For any $j \geq 2$, the interval $[a^{1/j},b^{1/j}]$ has size $O(N^c)$ (by the mean value theorem), and so the $j^{\operatorname{th}}$ summand on the RHS can be computed in time $O(N^{c+o(1)})$ by the AKS algorithm \cite{aks}.  Thus we see that to prove Theorem \ref{parity}, it will suffice to compute the quantity
$$ \sum_{a \leq n \leq b} 2^{\omega(n)}$$
in time $O(N^{1/2-c+o(1)})$.  Using \eqref{omn}, we can expand this expression as
\begin{equation}\label{dm}
 \sum_d \mu(d) \sum_{a/d^2 \leq m \leq b/d^2} \tau(m).
\end{equation}
Clearly $d$ can be restricted to be $O(N^{1/2})$.

We first dispose of the large values of $d$ in which $d > N^{0.49}$ (say).  Then $m = O(N^{0.02})$, so we can rearrange this portion of \eqref{dm} as
\begin{equation}\label{dam}
 \sum_{m = O(N^{0.02})} \sum_{\sqrt{a/m} \leq d \leq \sqrt{b/m}; d \ge N^{0.49}} \mu(d) \tau(m).
\end{equation}
For each value of $m$, there are $O(N^c)$ possible values of $d$, each of size $O(N^{1/2})$.  Each such $d$ can be factored using trial division in time $O(N^{1/4+o(1)})$ (or one can use more advanced factoring algorithms if desired), and so each of the $O(N^{0.02+c})$ summands can be computed in time $O(N^{1/4+o(1)})$, giving a net cost of $O(N^{0.27+c+o(1)})$ which is acceptable for $c$ small enough.

For the remaining values of $d$, we can use the sieve of Erathosthenes to factorise all the $d$ (and in particular, compute $\mu(d)$) in time $O(N^{0.49+o(1)})$.  So the main task is to compute the inner sum of \eqref{dm} for such $d$.

We will shortly establish

\begin{theorem}\label{dirich}  The expression $\sum_{n \leq x} \tau(n)$ can be computed in time $O(x^{1/2-c_0+o(1)})$ for some absolute constant $c_0>0$.
\end{theorem}

Assuming this for the moment, we see that for each $d \leq N^{0.49}$, the summand in \eqref{dm} can be computed in time $O( N^{o(1)} (N/d^2)^{1/2-c_0} )$.  Summing in $d$, we obtain a total time cost of $O( N^{1/2-c_0/10 + o(1)} )$ (say), which is acceptable if $c$ is chosen small enough depending on $c_0$.  

So it suffices to establish Theorem \ref{dirich}. The argument here is loosely inspired by the arguments used to establish the elementary bound $\sum_{n \leq x} \tau(n) = x \log x - (2\gamma - 1) x + O(x^{1/3+o(1)})$ in \cite[Chapter 3]{vin}.

Clearly we may shift $x$ to be a non-integer.  We then apply the Dirichlet hyperbola identity \eqref{hyperbola} (with $f=g=1$ and $y=\sqrt{x}$) to expand
$$ \sum_{n \leq x} \tau(n) = 2 \sum_{n \leq \sqrt{x}} \left\lfloor \frac{x}{n} \right\rfloor - \lfloor \sqrt{x} \rfloor^2.$$
It thus suffices to evaluate the integer
$$ \sum_{n \leq \sqrt{x}} \left\lfloor \frac{x}{n} \right\rfloor $$
in time $O(x^{1/2-c_0+o(1)})$.   In fact, we have

\begin{proposition}[Complexity of the hyperbola]\label{partition}  In time $O(x^{0.49+o(1)})$, one can obtain a partition of the discrete interval $\{n: 1 \leq n \leq \sqrt{x}\}$ into $O(x^{0.49+o(1)})$ arithmetic progressions, with the function $n \mapsto \left\lfloor \frac{x}{n} \right\rfloor$ linear on each arithmetic progression.
\end{proposition}

Since one can use explicit formulas to sum any linear function with coefficients of size $O(x)$ on an arithmetic progression of integers of size $O(x)$ in time $O(x^{o(1)})$, Theorem \ref{dirich} now follows immediately from the above proposition.

\begin{proof} By using the singleton sets $\{n\}$ to partition all the numbers less than $x^{0.49}$, we see that it suffices to partition the interval $\{ n: x^{0.49} \leq n \leq \sqrt{x}\}$.

Let $x^{0.49} \leq n_0 \leq \sqrt{x}$ be arbitrary, and set $Q := x^{0.1}$.  By the Dirichlet approximation theorem, there exist integers $1 \leq q \leq Q$ and $a \geq 1$ such that $|\frac{x}{n_0^2} - \frac{a}{q}| \leq \frac{1}{qQ}$.  These integers can be easily located in time $O(x^{o(1)})$ using continued fractions.  We now expand the quantity $\frac{x}{n}$ where $n = n_0 + l q + r$, $l \geq 0$, and $0 \leq r < q$.  Since
$$ \frac{1}{n_0 + y} = \frac{1}{n_0} - \frac{y}{n_0^2} + \frac{y^2}{n_0^2 (n_0+y)}$$
for any $y$, we have
$$ \frac{x}{n} = \frac{x}{n_0} - \frac{x (lq+r)}{n_0^2} + \frac{x (lq+r)^2}{n_0^2(n_0+y)}.$$
We expand $\frac{x}{n_0^2} = \frac{a}{q} + \frac{\theta}{qQ}$ for some explicitly computable $|\theta| \leq 1$, to obtain
$$ \frac{x}{n} = \frac{x}{n_0} - al - \frac{\theta l}{Q} - \frac{xr}{n_0^2} + \frac{x (lq+r)^2}{n_0^2(n_0+lq+r)}.$$
We thus have
$$ \left\lfloor \frac{x}{n} \right\rfloor = - al + \lfloor P(l) \rfloor$$
where $P = P_{x,n_0,a,q,\theta,r}$ is the rational function
$$ P(l) := \frac{x}{n_0} - \frac{xr}{n_0^2} - \frac{\theta l}{Q} + \frac{x (lq+r)^2}{n_0^2(n_0+lq+r)}.$$
The first two terms on the right-hand side are independent of $l$.  If we restrict $l$ to the range $0 \leq l \leq Q$, then the third term has magnitude at most $1$, and the fourth term has magnitude at most
$$ O( \frac{ x Q^4 }{n_0^3} ) = O( x^{-0.01} ).$$
Thus (for $x$ large enough) we see that $P$ fluctuates in an interval of length at most $3$, and so $\lfloor P(l) \rfloor$ takes at most three values. 
For any such value $k$, the set $\{ l: \lfloor P(l) \rfloor = k \}$ is a union of intervals, bounded by the sets $\{ l: P(l) = k \}$ and $\{ l: P(l) = k+1\}$.  As $P$ is a rational function in $l$ of bounded degree, we see from Bezout's theorem that these latter sets have cardinality $O(1)$, and so the set $\{ l: \lfloor P(l) \rfloor = k \}$ is the union of $O(1)$ intervals.  Furthermore, the endpoints of these intervals can be computed explicitly in time $O(x^{o(1)})$, by using the explicit formula for the solution of the cubic.  We conclude that in time $O(x^{o(1)})$, one can partition each arithmetic progression $\{ n_0+lq+r: 0 \leq l \leq Q \}$ for $0 \leq r < q$ into $O(1)$ subprogressions, with $n \mapsto \lfloor \frac{x}{n} \rfloor$ linear on each subprogression.  Performing this once for each residue class $r \hbox{ mod } q$, we see that in time $O(x^{o(1)} q)$, we can partition the interval $\{ n: n_0 \leq n < n_0 + qQ \}$ into $O(q)$ progressions, with $n \mapsto \left\lfloor \frac{x}{n} \right\rfloor$ linear on each progression.  If we apply this observation with $n_0$ set equal to the left endpoint of the interval 
$\{ n: x^{0.49} \leq n \leq \sqrt{x}\}$, we may partition an initial segment of this interval into progressions with the required linearity property.  Removing this initial segment, and iterating this procedure (updating $n_0$ and $q$ at each stage) we then obtain the claim.  (Note that if the interval $\{ n: n_0 \leq n < n_0 + qQ \}$ overflows beyond $\sqrt{x}$, then we may simply partition the remaining portion of the interval into singletons, at a cost of $O(x^{0.2})$ progressions.)
\end{proof}

\subsection{A refinement}

By modifying the above argument, one can in fact compute $\sum_{n \leq x} \tau(n)$ in $O(x^{1/3+o(1)})$ time, though this particular argument does not extend as easily to the polynomial setting as the one given above.  We sketch the details as follows.  As before, it suffices to compute
$$ \sum_{n \leq \sqrt{x}} \left\lfloor \frac{x}{n} \right\rfloor$$
in time $O(x^{1/3+o(1)})$.  By dyadically decomposing the interval $\{ n: n \leq \sqrt{x} \}$ into dyadic intervals $\{ x: A \leq n < 2A\}$ for various values of $A$, it suffices to compute
$$ \sum_{A \leq n < 2A} \left\lfloor \frac{x}{n} \right\rfloor$$
in time $O(x^{1/3+o(1)})$ for all $A \leq \sqrt{x}$.  We may assume that $A > 100 x^{1/3}$ since one can sum the series one term at a time otherwise.

We consider the subtask of computing a partial sum of the form
$$ \sum_{n_0 \leq n < n_0+q} \left\lfloor \frac{x}{n} \right\rfloor$$
where $A \leq n_0 < 2A$ and $q$ is chosen so that $|x/n_0^2 - a/q| \leq 1/qQ$ with $1 \leq q \leq Q$ and $a$ coprime to $q$ as above, where we now optimise $Q$ to equal $A x^{-1/3}$.  We claim that this sum can be computed in $O(x^{o(1)})$ time.

As this sum is an integer, it suffices to compute the sum with an error of less than $1/2$.  
Writing $n = n_0+r$ and $x/n_0^2 = a/q + \theta/qQ$ and expanding as before we have
$$ \frac{x}{n} = \frac{x}{n_0} - \frac{ar}{q} - \frac{\theta r}{qQ} + \frac{x r^2}{n_0^2 (n_0+r)}$$
and thus (for $0 \leq r < q$)
$$ \frac{x}{n} = \frac{x}{n_0} - \frac{ar}{q} + O\left( \frac{1}{q} \right)$$
where we have used the assumptions $q \leq Q = A x^{-1/3}$.

As $r$ runs from $0$ to $q-1$, the fractional parts of $\frac{ar}{q}$ take each of the values $\frac{0}{q}, \frac{1}{q}, \ldots, \frac{q-1}{q}$ exactly once, since $a$ is coprime to $q$.  We conclude that
$$ \left\lfloor  \frac{x}{n} \right\rfloor = \left\lfloor\frac{x}{n_0} - \frac{ar}{q}\right\rfloor$$
for all but $O(1)$ values of $r$, each of which can be computed explicitly in $O(x^{o(1)})$ time.  So we are left with computing
$$ \sum_{0 \leq r < q} \left\lfloor\frac{x}{n_0} - \frac{ar}{q}\right\rfloor = \sum_{0 \leq i < q} \left\lfloor \frac{x}{n_0} - \frac{i}{q} \right\rfloor$$
which can easily be computed in $O(x^{o(1)})$ time, and the claim follows.

A modification of the above argument shows that we can in fact compute $\sum_{n_0 \leq n < n_0+kq} \left\lfloor \frac{x}{n}\right\rfloor$ in $O(x^{o(1)})$ time whenever $kq = O(Q)$.  As such, we can compute the entire sum $\sum_{A \leq n < 2A} \left\lfloor \frac{x}{n} \right\rfloor$ in time $O( x^{o(1)} A / Q ) = O( x^{1/3+o(1)} )$ by summing in blocks of size $Q$, and the claim follows.

\section{The circuit complexity of the prime polynomial mod $2$}\label{circuit}

We now modify the above arguments to establish Theorem \ref{circo}.  We begin by showing a non-trivial gain in circuit complexity for a quadratic sum.

\begin{lemma}\label{tamlemma}  Let $a, b, c, q = O(N)$ be integers, then the expression
\begin{equation}\label{tam}
 \sum_{m=0}^{q-1} t^{am^2 + bm + c}
\end{equation}
has circuit complexity $O( N^{o(1)} q^{1-c_0} )$ in the polynomial ring $\Z[t]$ for some absolute constant $c_0>0$.
\end{lemma}

Note that this is a power saving over the trivial bound of $O(N^{o(1)} q )$ (note that by repeated squaring, any monomial $t^n$ with $n=O(N^{O(1)})$ has circuit complexity $O(N^{o(1)})$).

\begin{proof}  It suffices to establish this lemma when $q$ is a perfect cube $q=Q^3$, as the general case can then be established by approximating $q$ by the nearest cube and evaluating the remaining $O(q^{1/3})$ terms by hand.

Next, we expand $m$ in base $Q$ as $m = i + Qj + Q^2 k$ for $0 \leq i,j,k < Q$.  We can then expand $am^2 + bm + c$ as a quadratic polynomial in $i,j,k$, which we split as
$$ am^2 + bm + c = U(i,j) + V(j,k) + W(k,i)$$
for some explicit quadratic polynomials $U, V, W$, whose coefficients have polynomial size in $N$.  We can then express \eqref{tam} as
$$ \sum_{i=0}^{Q-1} \sum_{j=0}^{Q-1} \sum_{k=0}^{Q-1} t^{U(i,j)} t^{V(j,k)} t^{W(k,i)}$$
or more compactly as
$$ \operatorname{tr}( A B C )$$
where $A, B, C$ are the $Q \times Q$ matrices
$$ A := (t^{U(i,j)})_{0 \leq i,j < Q}; \quad B := (t^{V(i,j)})_{0 \leq j,k < Q}; \quad C := (t^{W(k,i)})_{0 \leq k,i < Q}.$$
Each of the matrices $A, B, C$ has a circuit complexity of $O(N^{o(1)} Q^2)$.  Using the Strassen fast matrix multiplication algorithm \cite{strassen}, one can multiply $A, B, C$ together using a circuit of complexity $O(Q^{3-c_0})$ for some absolute constant $c_0>0$.  Taking the trace requires another circuit of complexity $O(Q)$.  Putting all these circuits together and recalling that $Q = q^{1/3}$, one obtains the claim.
\end{proof}

It would be of interest to see if similar power savings can also be obtained for analogous sums in which the quadratic exponent $an^2+bn+c$ is replaced by a higher degree polynomial.  It may be that a generalisation of the Strassen algorithm to tensors would be relevant for this task.

Next, we need the following modification of Proposition \ref{partition}.

\begin{proposition}[Complexity of the hyperbola, II]\label{partition2}  There exists an absolute constant $c_0 > 0$ such that if $0 < c < c_0$ is sufficiently small, then for any $0 < x' < x$ with $x-x' \leq x^{1/2+c}$, and in time $O(x^{1/2 - c_0 + o(1)})$, one can obtain a partition of the discrete interval $\{n: x^{1/2-c} \leq n \leq \sqrt{x}\}$ into $O(x^{1/2-c_0+o(1)})$ arithmetic progressions, with the function $n \mapsto \lfloor \frac{x}{n} \rfloor$ linear on each arithmetic progression, and the function $n \mapsto \lfloor \frac{x}{n} \rfloor - \lfloor \frac{x'}{n} \rfloor$ is constant.
\end{proposition}

\begin{proof}  Let $c_0 > 0$ be a sufficiently small constant, and assume that $0 < c < c_0$ is sufficiently small as well.
Let $x^{1/2-c} \leq n \leq \sqrt{x}$ be arbitrary, and set $Q := x^{10c_0}$.  As in the proof of Proposition \ref{partition}, there exist integers $1 \leq q \leq Q$ and $a \geq 1$ such that $\frac{x}{n_0^2} = \frac{a}{q} + \frac{\theta}{qQ}$ for some $|\theta| \leq 1$, where $n = n_0 + lq + r$ and $0 \leq l,q,r \leq Q$.  Since $n \geq x^{1/2-c}$, we have (for $x$ large enough) that $n_0 \geq x^{1/2-c}/2$ (say).  A brief computation (noting that $|x-x'| \leq x^{1/2+c}$) then shows that $\frac{x'}{n_0^2} = \frac{a}{q} + \frac{\theta'}{qQ}$ for some $|\theta'| \leq 2$ if $c$ is small enough and $x$ is sufficiently large.  The claim then follows by repeating the proof of Proposition \ref{partition} (the main difference being that the rational function $P$ is now replaced by a pair $P, P'$ of rational functions).
\end{proof}

We now combine Lemma \ref{tamlemma} and Proposition \ref{partition2} to obtain

\begin{corollary}\label{code}  If $c > 0$ is sufficiently small, then for any $0 < a < b < N$ with $b-a \leq N^{1/2+c}$, the polynomial
$$ \sum_{a < n \leq b} \tau(n) t^n$$
has circuit complexity $O(N^{1/2-c+o(1)})$ for some absolute constant $c>0$.
\end{corollary}

\begin{proof} This is analogous to Theorem \ref{dirich}.  We let $c>0$ be a sufficiently small quantity to be chosen later.

We may assume that $a, b$ are not integers.  We expand this polynomial as
$$ \sum_{n, m \geq 1: a < nm \leq b} t^{nm}.$$
Observe that if $a < nm \leq b$, then one either has $1 \leq n \leq \sqrt{b}$ or $1 \leq m \leq \sqrt{b}$, or both, with the last case occuring precisely when $a/\sqrt{b} \leq n \leq \sqrt{b}$ and $a/n \leq m \leq \sqrt{b}$.  In the first case, we rewrite the condition $a < nm \leq b$ as $a/n < m \leq b/n$; in the second case, we rewrite that condition as $a/m < n \leq b/m$.  After swapping $n$ and $m$ in the second case, we can rearrange the above polynomial as
$$ 2 \sum_{1 \leq n \leq \sqrt{b}} \sum_{a/n < m \leq b/n} t^{nm} - \sum_{a/\sqrt{b} \leq n \leq \sqrt{b}} \sum_{a/n \leq m \leq \sqrt{b}} t^{nm}.$$
The second sum contains only $O(N^{2c})$ terms and so can easily be verified to have a circuit complexity of $O(N^{2c+o(1)})$, which is acceptable.  So it will suffice to show that the sum
\begin{equation}\label{nah}
 \sum_{1 \leq n \leq \sqrt{b}} \sum_{\lfloor a/n \rfloor + 1 \leq m \leq \lfloor b/n \rfloor} t^{nm}
\end{equation}
has circuit complexity $O(N^{1/2-c+o(1)})$.  

Using the geometric series formula, the inner sum has circuit complexity $O(N^{o(1)})$ for each fixed $n$.  This is already sufficient to dispose of the contribution of the terms in \eqref{nah} for which $n \leq N^{1/2-c}$, so it remains to bound the circuit complexity of
$$
 \sum_{N^{1/2-c+o(1)} \leq n \leq \sqrt{b}} \sum_{\lfloor a/n \rfloor + 1 \leq m \leq \lfloor b/n \rfloor} t^{nm}.$$

Using Proposition \ref{partition2} and in time $O(N^{1/2-c_0+o(1)})$ for some absolute constant $c_0$ (independent of $c$), we may partition $\{n: N^{1/2-c_0+o(1)} \leq n \leq \sqrt{b}\}$ into arithmetic progressions $P_1,\ldots,P_k$ with $k = O(N^{1/2-c_0+o(1)})$, such that $\lfloor b/n \rfloor$ is a linear function of $n$ on each of these progressions, and $\lfloor a/n \rfloor - \lfloor b/n \rfloor$ is constant.  This constant is of size $O(N^{2c})$.  
  Applying Lemma \ref{tamlemma} (after first switching the order of summation), the sum 
$$\sum_{n \in P_j} \sum_{\lfloor a/n \rfloor + 1 \leq m \leq \lfloor b/n \rfloor} t^{nm}$$
has a circuit complexity of $O( N^{2c+o(1)} |P_j|^{1-c_1} )$ for some $c_1>0$, so that \eqref{nah} has a circuit complexity of
$$ O(N^{1/2-c_0+o(1)}) + \sum_{j=1}^k O( N^{2c+o(1)} |P_j|^{1-c_1} ).$$
By H\"older's inequality, one has
$$ \sum_{j=1}^k |P_j|^{1-c_1} \leq (\sum_{j=1}^k |P_j|)^{1-c_1} k^{c_1}.$$
Since $\sum_{j=1}^k |P_j| = O(N^{1/2})$ and $k = O(N^{1/2-c_0+o(1)})$, we obtain a total circuit complexity bound of
$$ O( N^{1/2 - c_0 c_1 + 2c + o(1)} )$$
and the claim follows if $c$ is chosen sufficiently small.
\end{proof}

Now we can prove Theorem \ref{circo}.  We repeat the arguments from the previous section.  First observe that
$$ \sum_{a \leq n \leq b} 2^{\omega(n)} t^n \equiv 2 P_{a,b}(t) + 2 \sum_{j=2}^\infty \sum_{a^{1/j} \leq p \leq b^{1/j}} t^{p^j} \hbox{ mod } 4.$$
Because $b-a = O(N^{1/2+c})$ and $b,a$ are comparable to $N$, we see from the mean value theorem that $b^{1/j}-a^{1/j} = O(N^c)$ for all $j \geq 2$.  We thus see that the total number of primes $p$ in the latter sum are $O(N^{c+o(1)})$ on the right-hand side, and so this sum has a circuit complexity of $O(N^{c+o(1)})$.  Thus it suffices to show that the polynomial
$$ \sum_{a \leq n \leq b} 2^{\omega(n)} t^n \hbox{ mod } 4$$
has circuit complexity $O(N^{1/2-c+o(1)})$.  Using \eqref{omn}, we rewrite this polynomial as
\begin{equation}\label{dm2}
 \sum_d \mu(d) \sum_{a/d^2 \leq m \leq b/d^2} \tau(m) t^{d^2 m}.
\end{equation}
Clearly $d$ can be restricted to be $O(N^{1/2})$.

Once again, we first dispose of the large values of $d$ in which $d > N^{0.49}$.  This portion of \eqref{dm2} can be rearranged as
$$ \sum_{m = O(N^{0.02})} \sum_{\sqrt{a/m} \leq d \leq \sqrt{b/m}; d \ge N^{0.49}} \mu(d) \tau(m) t^{dm}.$$
Repeating the arguments from the previous section (and specifically, the arguments used to compute \eqref{dam}), this term can be given a circuit complexity of $O(N^{0.27+c+o(1)})$.

For the remaining values of $d$, we again use the sieve of Erathosthenes to compute all the $\mu(d)$ in time $O(N^{0.49+o(1)})$.  Using Lemma \ref{code}, each instance of the inner sum  $\sum_{a/d^2 \leq m \leq b/d^2} \tau(m) t^{d^2 m}$ has a circuit complexity of $O( (N/d^2)^{1/2-c_0+o(1)} )$ for some absolute constant $c_0 > 0$.  Summing in $d$ as before, we obtain a total circuit complexity of
$$ O(N^{0.49+o(1)} + \sum_{d \leq N^{0.49}} O( (N/d^2)^{1/2-c_0+o(1)} )$$
which sums to $O(N^{1/2-c+o(1)})$ as desired, for $c$ small enough.

\section{Possible extensions}\label{extend}

The circuit complexity bound on the prime polynomial $P_{a,b}(t)$ given by Theorem \ref{circo} lets us compute $P_{a,b}(t) \hbox{ mod } (2, g)$ in time $O(N^{1/2-c/2+o(1)})$ for any polynomial $g \in {\bf F}_2[t]$ of degree $O(N^{c/4})$, if $c > 0$ is sufficiently small.  Unfortunately, this is not strong enough to deterministically determine in time $O(N^{1/2-c/2+o(1)})$ whether $P_{a,b}(t)$ is non-trivial or not, although as mentioned in the introduction it at least gives a bounded-error probabilistic test in this amount of time.  It may be however that by using additional algorithms (such as the Fast Fourier Transform, or the multipoint polynomial evaluation algorithm of Borodin and Moenk\cite{boro}) one may be able to compute quantities such as $P_{a,b}(t) \hbox{ mod } (2, g)$ for multiple values of $g$ simultaneously in $O(N^{1/2-c/2+o(1)})$ time, or perhaps variants such as $P_{a,b}(t^j) \hbox{ mod } (2, g)$.  However, it is \emph{a priori} conceivable (though very unlikely) that the degree $O(N)$ polynomial $P_{a,b}(t) \hbox{ mod } 2$ is divisible by as many as $O(N^{1-c/4})$ irreducible polynomials $g \in {\bf F}_2[t]$ of degree $O(N^{c/4})$, so it is not yet clear to us how to use this sort of test to deterministically settle the decision problem in $O(N^{1/2-c+o(1)})$ time.  One possibility would be to find a relatively small set of $g$ for which it was not possible for $P_{a,b}(t) \hbox{ mod } 2$ to be simultaneously divisible by, without vanishing entirely.  Note that a somewhat similar idea is at the heart of the AKS primality test \cite{aks}.

If one could compute $\pi(x) \hbox{ mod } q$ (or $\pi(b)-\pi(a) \hbox{ mod } q$) for each prime $1 \leq q \leq O(\log N)$ in time $O(N^{1/2-c+o(1)})$ uniformly in $q$, where $x,a,b = O(N)$, then from the Chinese remainder theorem we could compute $\pi(x)$ or $\pi(b)-\pi(a)$ itself in time $O(N^{1/2-c+o(1)})$, thus solving Conjecture \ref{break}.  The above analysis achieves this goal for $q=2$.  However, the methods deteriorate extremely rapidly in $q$.  For instance, if one wished to compute $\pi(x) \hbox{ mod } 3$ by the above methods, one would soon be faced with understanding the sum 
$$\sum_{n<x} \tau_2(n) = \sum_{a,b,c \geq 1: abc \leq x} 1$$
where $x=O(N)$ and $\tau_2(n) := \sum_{a,b,c: abc = n} 1$ is the second divisor function.  (Observe that the expression $\sum_{d: d^3|n} \mu(d) \tau_2(n/d^3)$ is divisible by $9$ unless $n$ is equal to a $1$ or a power of a prime.)  The three-dimensional analogue of the Dirichlet hyperbola method allows one to evaluate this expression in time $O(N^{2/3+o(1)})$.  The type of arguments used in the previous sections would reduce cost this slightly to $O(N^{2/3-c+o(1)})$ for some small $c>0$ but this is inferior to the bound $O(N^{1/2+o(1)})$ that can already be obtained for $\pi(x)$.

We have not attempted to optimise the exponent savings $c>0$ appearing in the results of this paper.  It may be that improvements to these exponents may be obtained by making more accurate approximations of the discrete hyperbola $\{ (n,\lfloor \frac{x}{n} \rfloor): 1 \leq n \leq \sqrt{x}\}$ than the piecewise linear approximation given by Lemma \ref{partition}; for instance, piecewise polynomial approximations may ultimately be more efficient.

It may also be of interest to obtain circuit complexity bounds for more general expressions than the prime polynomial $\sum_{a \leq p \leq b} t^p$; for instance one could consider $\sum_{a \leq p \leq b} t^{p^2}$ or more generally $\sum_{a \leq p \leq b} t^{h(p)}$ for some fixed polynomial $h$.

Some progress along the above lines will appear in forthcoming work of Croot, Hollis, and Lowry (in preparation).


\begin{thebibliography}{10}

\bibitem{aks}
M. Agrawal, N. Kayal, N. Saxena, \emph{PRIMES is in P}, Annals of Mathematics \textbf{160} (2004), no. 2, pp. 781–-793.

\bibitem{bhp}
R. C. Baker, G. Harman, J. Pintz, \emph{The difference between consecutive primes}, II, Proceedings of the London Mathematical Society \textbf{83}, (2001), 532-–562.

\bibitem{boro}
A. Borodin, R. Moenk, \emph{Fast Modular Transforms}, Jour. of Comp. and System Sciences, \textbf{8} (1974), 366--386.
 
\bibitem{cramer}
H. Cram\'er, \emph{On the order of magnitude of the difference between consecutive prime numbers}, Acta Arithmetica 2 (1936), 23–-46.

\bibitem{rivat}
M. Deleglise, J. Rivat, \emph{Computing $\pi(x)$: the Meissel, Lehmer, Lagarias, Miller, Odlyzko method}, Math. Comp. Vol. 65 (1996), 235--245.

\bibitem{granville}
A. Granville, \emph{Harald Cram\'er and the distribution of prime numbers}, Scandinavian Actuarial Journal 1(1995), 12–-28.

\bibitem{heath}
D. R. Heath-Brown, \emph{Primes represented by $x^3+2y^3$}. Acta Math. 186 (2001), no. 1, 1--84.
 
\bibitem{lmo}
J. C. Lagarias, V. S. Miller, A. M. Odlyzko, \emph{Computing $\pi(x)$: The Meissel-Lehmer method},
Math. Comp. 44 (1985), 537--560. 

\bibitem{lag}
J. C. Lagarias, A. M. Odlyzko, \emph{Computing $\pi(x)$: An analytic method}, J. Algorithms 8 (1987), 173--191.

\bibitem{pss}
J. Pintz, W. Steiger, E. Szemer\'edi, \emph{Infinite sets of primes with fast primality tests and quick generation of large primes}, Math. Comp. \textbf{53} (1989), no. 187, 399–-406.

\bibitem{polywiki} D.H.J. Polymath, {\tt michaelnielsen.org/polymath1/index.php?title=Polymath1}

\bibitem{strassen}
V. Strassen, \emph{Gaussian elimination is not optimal}, Numer. Math. 13 (1969), 354--356.

\bibitem{tenenbaum}
G. Tenenbaum, Introduction to analytic and probabilistic number theory. Translated from the second French edition (1995) by C. B. Thomas. Cambridge Studies in Advanced Mathematics, 46. Cambridge University Press, Cambridge, 1995.
 
\bibitem{vin}
I. M. Vinogradov, Elements of Number Theory, Mineola, NY: Dover Publications, 2003, 


\end{thebibliography}
\end{document}